\documentclass[11pt,reqno]{amsart}
\usepackage[T1]{fontenc}
\usepackage{lmodern}
\usepackage[english]{babel}
\usepackage{graphicx}
\usepackage{amsmath}
\usepackage{amssymb}
\usepackage{mathrsfs}
\usepackage{bookmark}
\usepackage[numbers,sort&compress]{natbib}
\numberwithin{equation}{section}
\usepackage{amsthm}
\usepackage[a4paper]{geometry}
\newcommand\Item[1][]{%
  \ifx\relax#1\relax  \item \else \item[#1] \fi
  \abovedisplayskip=0pt\abovedisplayshortskip=0pt~\vspace*{-\baselineskip}}

\theoremstyle{plain}
 \newtheorem{theorem}{Theorem}[section]

\newtheorem{lemma}[theorem]{Lemma}
 
\theoremstyle{definition}

 \newtheorem{remark}[theorem]{Remark}
 \numberwithin{equation}{section}

\usepackage{caption}
\usepackage{tcolorbox}

\newcommand{\CA}{\mathcal A}			\newcommand{\CB}{\mathcal B}
\newcommand{\CC}{\mathcal C}			\newcommand{\calD}{\mathcal D}

			\newcommand{\CL}{\mathcal L}
\newcommand{\CM}{\mathcal M}			
			\newcommand{\CP}{\mathcal P}
			
			\newcommand{\CT}{\mathcal T}

				\newcommand{\BF}{\mathbb F}

				\newcommand{\BZ}{\mathbb Z}

\DeclareMathOperator{\Map}{Map}

\DeclareMathOperator{\Homeo}{Homeo}
\DeclareMathOperator{\Out}{Out}
\DeclareMathOperator{\Aut}{Aut}
\DeclareMathOperator{\Stab}{Stab}

\DeclareMathOperator{\D}{\partial}
\DeclareMathOperator{\Pol}{Pol}

\DeclareMathOperator{\const}{const}

\usepackage{mathtools}

\usepackage{pinlabel}
\usepackage{xcolor}

\definecolor{mygreen}{rgb}{0.0, 0.5, 0.0}
\definecolor{deepmagenta}{rgb}{0.8, 0.0, 0.8}
\definecolor{deepcarrotorange}{rgb}{0.91, 0.41, 0.17}
\definecolor{heartgold}{rgb}{0.5, 0.5, 0.0}

\newcommand{\xuparrow}[1]{%
  {\left\uparrow\vbox to #1{}\right.\kern-\nulldelimiterspace}
}
\definecolor{kellygreen}{rgb}{0.3, 0.73, 0.09}

\usepackage{subfigure}
\usepackage{BOONDOX-cal}
\usepackage{enumitem}
\usepackage{appendix}
\graphicspath{{images/}}

\begin{document}

\title[Deciding when two curves are of the same type]{Deciding when two curves are of the same type}

\author[Juan Souto]{Juan Souto}
\address{UNIV RENNES, CNRS, IRMAR - UMR 6625, F-35000 RENNES, FRANCE}
\email{jsoutoc@gmail.com}

\author[Thi Hanh Vo]{Thi Hanh Vo}
\address{DEPARTMENT OF MATHEMATICS, UNIVERSITY OF LUXEMBOURG}
\email{thihanh.vo@uni.lu}

\thanks{The second author was supported by ANR/FNR project SoS, INTER/ANR/16/11554412/SoS, ANR-17-CE40-0033 and the travel grant DPMA 1054-07-Z}
\date{\today}

\begin{abstract} 
Given two closed curves in a surface, we propose an algorithm to detect whether they are of the same type or not.
\end{abstract}

\maketitle

%\tableofcontents

\section{Introduction}

Whitehead's algorithm \cite{MR1503309} serves to determine if two elements $\gamma$ and $\eta$ in a finitely generated non-abelian free group $\BF$ differ by an automorphism of the latter, that is if there is $\phi\in\Aut(\BF)$ with $\phi(\gamma)=\eta$. An algorithm solving the same problem for surface groups is due to Levitt and Vogtmann \cite{MR1783856}. Both algorithms use different methods, but the basic strategy is surprisingly similar, consisting of two separate steps.

In a first step one brings the involved elements into a suitable "normal form". In the case of the free group one uses Whitehead automorphisms to reduce step by step their lengths. In the case of  surface groups one conjugates the given elements so that they are represented by paths with minimal self-intersection number. In either case the time complexity of this first step is at worst quadratic: the complexity of the involved elements is strictly reduced at every step and one stops when it is no longer possible to do so.

Once we are done with step one we go into the second step. Here the goal is to decide if two elements in normal form belong to the same connected component of the graph whose vertex set is the set of elements in normal form and where two vertices are joined if they differ by an elementary move: a Whitehead automorphism in the case of the free group and, in the case of the surface group, sliding a strand over a double point followed by a self-homeomorphism of the surface. Note that although the graph is infinite, the algorithm terminates because the function which associates to each vertex its complexity is constant on the connected components of the graph, meaning that the involved connected components are finite. However, the number of elements in normal form and with given complexity is exponential, meaning that a priori the second step has exponential time complexity. In fact, it seems that the running time of the Whitehead algorithm is only known to be polynomial in the case that $n=2$ \cite{Khan}. We do not know of any efficient estimate on the running time of the Levitt-Vogtmann algorithm.

The goal of this note is to give a different algorithm, one of polynomial time complexity, to decide if two elements of a surface group differ by an automorphism:

\begin{theorem}\label{thm-main1}
Let $\Sigma$ be a closed orientable connected surface of genus $g\ge 2$. There is a polynomial time algorithm which decides if two elements $\gamma,\eta\in\pi_1(\Sigma,*)$ differ by a group automorphism and which, if possible, finds $\phi\in\Aut(\pi_1(\Sigma,*))$ with $\phi(\gamma)=\eta$.
\end{theorem}

In the statement of Theorem \ref{thm-main1}, polynomial time means that the running time is bounded from above by a polynomial function in $\max\{\vert\gamma\vert,\vert\eta\vert\}$ where $\vert\cdot\vert$ stands for the word length with respect to any finite generating set.

Note that the conjugacy problem in the surface group $\pi_1(\Sigma,*)$ has a linear time solution because the latter is hyperbolic \cite{Holt}. This means that to prove Theorem \ref{thm-main1} it is enough to give an algorithm deciding if there is an outer automorphism $\phi\in\Out(\pi_1(\Sigma,*))$ which maps the conjugacy class of $\gamma$ to that of $\eta$. From this point of view, it is easy to frame the whole problem in topological rather than algebraic terms. For once, conjugacy classes of elements in the fundamental group correspond to free homotopy classes of curves. But not only that. We can namely also identify $\Out(\pi_1(\Sigma,*))$ with the (full) mapping class group
$$\Map^\pm(\Sigma) = \Homeo(\Sigma) / \Homeo_0(\Sigma)$$
of $\Sigma$, where $\Homeo(\Sigma)$ is the group of self-homeomorphisms of $\Sigma$ and $\Homeo_0(\Sigma)$ is its identity component. In those terms we can basically rephrase Theorem \ref{thm-main1} as asserting that there is a polynomial time algorithm deciding if two curves are of the {\em same type}, meaning that the corresponding free homotopy classes differ by a mapping class.
The following is then our main result:

\begin{theorem}\label{thm-main}
Let $\Sigma$ be a compact orientable connected surface whose genus $g\ge 2$ and number of boundary components $b$, satisfy $2-2g-b<0$. There is a polynomial time algorithm which decides if two elements $\gamma,\eta\in\pi_1(\Sigma,*)$ are of the same type or not, finding if possible some $\phi\in\Map^\pm(\Sigma)$ with $\eta=\phi(\gamma)$.
\end{theorem}

The following is the basic strategy of the proof of Theorem \ref{thm-main}, at least in the central case that $\gamma$ and $\eta$ are primitive and fill. We produce in polynomial time a collection $\CC$ of mapping classes with the following properties:
\begin{itemize}
\item Any mapping class $\phi$ with $\phi(\gamma)=\eta$ belongs to $\CC$.
\item There are polynomially many elements in $\CC$ and for each one of them it takes polynomial time to decide if $\phi(\gamma)=\eta$.
\end{itemize}
Once we have this list then we just have to check for each $\phi\in\CC$ if $\phi(\gamma)=\eta$. If we find one then $\gamma$ and $\eta$ are of the same type. If not, then not.

\begin{remark}
As it is evident from the given sketch, our algorithm actually solves a more general problem than the one we are discussing here. Indeed, if $G\subset\Map^\pm(\Sigma)$ is an arbitrary subgroup whose membership problem is solvable in polynomial time, then we can decided again in polynomial time if two filling curves $\gamma$ and $\eta$ differ by an element of $G$: our algorithm gives in polynomial a complete list of all those $\phi\in\Map^\pm(\Sigma)$ with $\phi(\gamma)=\eta$, meaning that we just have to check one after the other if any of them belongs to $G$. For $\gamma$ and $\eta$ non-filling, things are a bit more complicated but it all boils down then to be able to check in polynomial time if a subsurface of $\Sigma$ is in the $G$-orbit of another subsurface. The reader will have no difficulty filling in the details after studying the proof of Theorem \ref{thm-main}. 
\end{remark}

This paper is organised as follows. In section \ref{sec:preli} we fix some terminology that we will use throughout the paper and prove a lemma explicitly relating two notions of length of a curve: length with respect to a triangulation and intersection number with a filling curve. 

In section \ref{sec: swiss guy} we present some basic things that one can do in polynomial time, things like for example checking if two curves are freely homotopic or detecting if two simple multicurves are of the same type. 

Armed with all these tools we prove Theorem \ref{thm-main} in section \ref{proof-thm-main}. In the proof we first add additional assumptions such as for example that the curves are primitive, dealing with the remaining cases at the end of the proof. Once Theorem \ref{thm-main} has been proved, Theorem \ref{thm-main1} follows easily. 

In section \ref{sec: dehn-thurston} we then sketch a small variation of the proof of Theorem \ref{thm-main} (for filling curves) replacing the triangulations we will have been using until that point by Dehn-Thurston coordinates. Working with Dehn-Thurston coordinates is more involved than working with triangulations---and this is why we choose the latter for the bulk of the paper---but it gives slightly tighter results. For example we get a polynomial of degree $18g+6b-14$ bounding the complexity of the algorithm.

\subsection*{Acknowledgements.} The first author would like to thank Monika Kudlinska for explaining to him her algorithm \cite{Monika} to detect if a curve is filling. The present paper benefited very much from Kudlinska's work.
The second author thanks IRMAR, Universit\'e de Rennes 1 for its hospitality while this work was initiated.
She would also like to thank Hugo Parlier for making this collaboration happen.

\section{Preliminaries}\label{sec:preli}

Let from now on $\Sigma$ be as in Theorem \ref{thm-main}, that is a compact orientable connected surface of genus $g\ge 2$ and $b$ boundary components satisfying $2-2g-b<0$. Let us also fix $*\in\Sigma$ a base-point and let $\vert\cdot\vert$ stand for the word length of elements in $\pi_1(\Sigma,*)$ with respect to some finite generating set. 

\subsection{Terminology}

Throughout this note we will tacitly identify free homotopy classes of un-oriented curves with actual representatives, referring to both as {\em curves}. In particular, we will also identify curves with the conjugacy classes of elements (and their inverses) in $\pi_1(\Sigma)$. A closed curve is \textit{essential} if it is neither homotopically trivial nor homotopic to a boundary component, and it is {\em primitive} if it is not homotopic to any proper power of another curve.  A primitive essential curve is {\em simple} if the corresponding free homotopy class has representatives which are embedded. A {\em multicurve} is a finite union of distinct essential primitive curves. If the components of a multicurve are simple and have disjoint representatives, then the multicurve is a {\em simple multicurve}. A multicurves is {\em ordered} if its components are ordered. Ordered multicurves will be decorated with an arrow on top, for example $\vec\gamma$ or $\vec P$. The prime example of a simple multicurve, ordered of not, is a pants decomposition, that is a collection of $3g-3+b$ disjoint essential simple curves whose complement consists of $2g-2+b$ pairs of pants. A multicurve $\gamma$ is \textit{filling} if we have $\iota(\gamma,\eta)>0$ for every essential curve $\eta$, where $\iota(\cdot,\cdot)$ stands for the geometric intersection number. Filling multicurves are never simple, but the union of simple curves might well be filling.

The mapping class group $\Map^{\pm}(\Sigma)$ acts on the sets of curves, multicurves, simple curves, ordered simple multicurves, etc... and we say that two such are \textit{of the same type} if they are in the same $\Map^{\pm}(\Sigma)$-orbit. Note that if two ordered pants decompositions $\vec P$ and $\vec P'$ are of the same type then there are many mapping classes sending one to the other. Indeed, if $\phi\in\Map^{\pm}(\Sigma)$ is such that $\phi(\vec P)=\vec P'$ then the set of those mapping classes $\psi$ with $\psi(\vec P)=\vec P'$ agrees with the set of those of the form $\phi\circ\sigma$ where $\sigma\in\Stab_{\Map^{\pm}(\Sigma)}(\vec P)$. The stabiliser of an ordered pants decomposition $\vec P$ admits a simple description, mostly if we restrict our attention to the {\em pure mapping class group}
$$P\Map(\Sigma)=\left\{\phi\in\Map^{\pm}\ \middle\vert\begin{array}{l}
\text{ orientation preserving and mapping }\\
\text{ each boundary component to itself}\end{array}\right\}.$$
Indeed, the stabiliser in the pure mapping class group of an ordered pants decomposition $\vec P$ is nothing other than the subgroup generated by the Dehn twists $T_{P_i}$ along the components $P_1,\dots,P_{3g-3+b}$ of $\vec P$, together with the center $C(\Map(\Sigma))$ of the mapping class group. The center $C(\Map(\Sigma))$ is trivial unless the pair $(g,b)$ consisting of genus and number of boundary components of $\Sigma$ is one of the following: $(1,1)$, $(1,2)$ or $(2,0)$. In those cases its only non-trivial element is the hyper-elliptic involution. We refer the reader to \cite{MR2850125} for background on the mapping class group.

\subsection{Curves in normal form with respect to a triangulation}
Suppose now that $\CT$ is a triangulation of $\Sigma$, and denote its set of edges by $E(\CT)$ and its 1-skeleton by $\CT^{(1)}$. A closed immersed 1-dimensional manifold $\lambda$ in $\Sigma$ is in {\em normal form} with respect to $\CT$, or just $\CT$-normal, if it avoids all the vertices of $\CT$, is transversal to the edges, and intersects each 2-simplex $\Delta$ in a collection of arcs, whose vertices are in different edges of $\Delta$. The {\em $\CT$-length} of a multicurve $\gamma$ in $\Sigma$ is then defined to be smallest possible number that a $\CT$-normal representative of the free homotopy class $\gamma$ meets the edges of the triangulation. Using slightly different words:
$$\ell_\CT(\gamma)=\min\{\vert\alpha\cap\CT^{(1)}\vert\text{ where }\alpha\text{ is a }\CT\text{-normal representative of }\gamma\}.$$
We prove next the following relation between $\CT$-lengths and intersection numbers:

\begin{lemma}\label{mainlem}
If $\alpha$ is a filling multicurve in $\Sigma$ then we have
$$\ell_\CT(\beta)\le\ell_{\CT}(\alpha)\cdot\iota(\beta,\alpha)\cdot(1+\iota(\beta,\beta))$$
for every other multicurve $\beta$ in $\Sigma$. If moreover $\Sigma$ is closed then the statement remains true even after deleting the factor $1+\iota(\beta,\beta)$.
\end{lemma}
\begin{proof}
Abusing notation we might identify $\alpha$ and $\beta$ with suitably chosen representatives. In other words, we might suppose that $\alpha$ and $\beta$ are actual curves instead of merely free homotopy classes, that they are in $\CT$-normal form and in general position with respect to each other, and that moreover we have
$$\vert\alpha\cap\beta\vert=\iota(\alpha,\beta)\text{, }\ell_{\CT}(\alpha)=\vert\alpha\cap\CT\vert\text{ and }\ell_{\CT}(\beta)=\vert\beta\cap\CT\vert.$$
Note that the assumption that $\alpha$ fills means that each component $\Delta$ of $\Sigma\setminus\alpha$ is either a disk or an annulus containing a boundary component. If $\Sigma$ is closed then only disks are possible.

Note also that the boundary of each such disk $\D\Delta$ meets $\CT$ at most $2\ell_{\CT}(\alpha)$ times. Now, the curve $\alpha$ cuts $\beta$ into segments $I_1,\dots,I_r$ where $r=\iota(\alpha,\beta)$. Each such segment $I_i$ is contained in a connected component $\Delta$ of $\Sigma\setminus\alpha$ and can be homotoped relative to its endpoints into $\D\Delta$. Since $\beta$ is supposed to have the minimal number of intersections with $\CT$ we deduce that
$$\vert I_i\cap\CT\vert\le\frac 12\vert\D\Delta\cap\CT\vert\le\ell_\CT(\alpha)$$
if $\Delta$ is a disk. If $\Delta$ is an annulus then $I_i$ could wrap a few times around, but wrapping around forces its self-intersection number, and thus that of $\beta$ to grow, meaning that we have
$$\vert I_i\cap\CT\vert\le\vert\D\Delta\cap\CT\vert\cdot(1+\iota(\beta,\beta))\le\ell_\CT(\alpha)\cdot(1+\iota(\beta,\beta))$$
It follows that
$$\ell_\CT(\beta)=\vert\beta\cap\CT\vert=\sum_{i=1}^{\iota(\alpha,\beta)}\vert I_i\cap\CT\vert\le\ell_\CT(\alpha)\cdot\iota(\alpha,\beta)$$
if the surface is closed and that 
$$\ell_\CT(\beta)=\vert\beta\cap\CT\vert=\sum_{i=1}^{\iota(\alpha,\beta)}\vert I_i\cap\CT\vert\le\ell_\CT(\alpha)\cdot\iota(\alpha,\beta)\cdot(1+\iota(\beta,\beta))$$
in general.
\end{proof}

Suppose now that we are given a triangle $\Delta$ with sides $e_1,e_2,e_3$ and recall that whenever we have three integers $m_1,m_2,m_3\ge 0$ satisfying
\begin{equation}\label{eq-star}
m_1+m_2+m_3\in 2\BZ, \ 
m_1\le m_2+m_3, \
m_2\le m_3+m_1, 
\text{ and }
m_3\le m_1+m_2, 
\end{equation}
then there is an, up to isotopy, unique configuration $J(m_1,m_2,m_3)$ consisting of $\frac{m_1+m_2+m_3}2$ disjoint simple arcs in $\Delta$, each one with endpoints in different sides, and such that $J$ meets $e_i$ exactly $m_i$ times. 

%%%%%%%%%%%%%%%%%%
\begin{figure}[h]
\labellist
\small\hair 2pt
\pinlabel {$m_3 = 3$} at 540 320
\pinlabel {$m_1 = 4$} at 0 320
\pinlabel {$m_2 = 5$} at 280 70
\endlabellist
\centering
\includegraphics[width=0.5\textwidth]{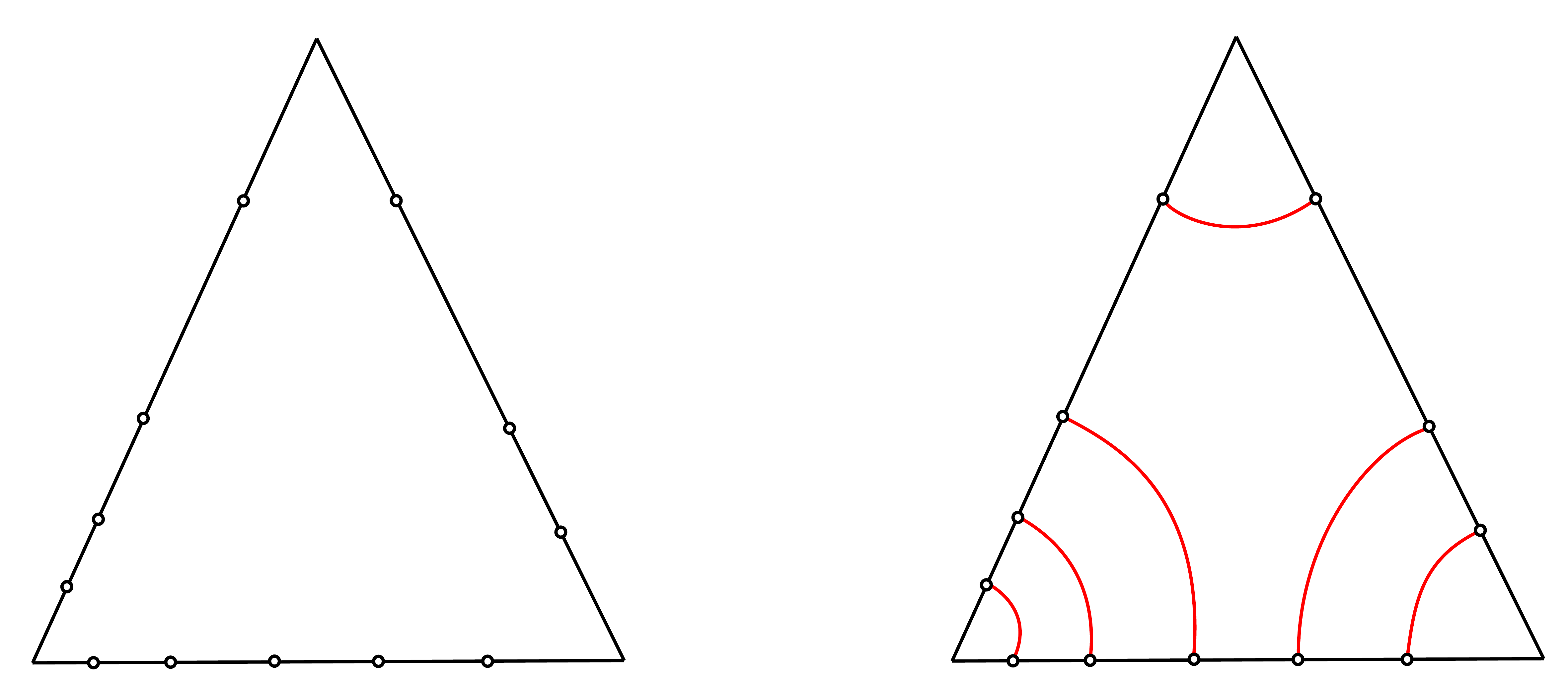}
\caption{Disjoint simple arcs in a triangle $\Delta$.}
\end{figure}       
%%%%%%%%%%%%%%%%%%

Returning to our triangulation $\CT$ let $\CA=\CA(\CT)$ be the set of vectors $\vec m=(m_e)_{e\in E(\CT)}\in\BZ_{\ge 0}^{E(\CT)}$ such that the triple $(m_{e_i},m_{e_j},m_{e_k})$ satisfies \eqref{eq-star} whenever $e_i$, $e_j$, and $e_k$ are the three edges of a triangle in $\CT$. We say that elements in $\CA$ are {\em admissible}. 

Suppose that $\vec m\in\CA$ is admissible. Since \eqref{eq-star} are satisfied for every triangle in $\CT$ we get an arc system $J(m_{e_i},m_{e_j},m_{e_k})$ for each triangle in $\CT$ (with sides $e_i$, $e_j$, and $e_k$). Gluing all those arc-systems to each other we get a simple, $\CT$-normal, embedded 1-manifold $\gamma_{\vec m}$ with
\begin{equation}\label{I am sick of this}
\vert\gamma_{\vec m}\cap e\vert=m_e
\end{equation}
for every edge $e\in E(\CT)$ of $\CT$.

\section{Some things we can do in polynomial time}\label{sec: swiss guy}
Continue with the same notation, suppose that we are given a triangulation $\CT$ of $\Sigma$ and a system of generators for the fundamental group $\pi_1(\Sigma,*)$. We think of the generators as being represented by actual $\CT$-transversal curves. 

We recall/discuss now some procedures for which there are polynomial time algorithms. It is very important to keep in mind that that we can run those algorithms inside other such algorithms and still have polynomial time: $P(Q(X))$ is a polynomial in $X$ if $P(X)$ and $Q(X)$ are. We will also use $\Pol(X)$ to denote a polynomial in $X$, but we will allow the concrete polynomial to change over time. For example, $\Pol(\Pol(X))=\Pol(X)$ is a perfectly sound statement. Similarly for polynomials of several variables. All polynomials under consideration will have positive entries and will be only fed positive quantities.

\begin{remark}
We should now point out that while having implementable algorithms is nice and dandy, this is not the standard we are aiming at. For us, an algorithm is just a clear procedure that your chosen, not extraordinarily creative but very reliable, punctual, clean, rule following and decently trained Swiss worker would be able to implement in polynomial time. It goes without saying that he or she would be equipped with illimited amounts of supplies like paper or multicolour pens. For example, to check if the multicurve $\gamma_{\vec m}$ corresponding to an admissible vector $\vec m\in\BZ_{\ge 0}^{E(\CT)}$ is connected one could ask our personal assistant to {\em first please draw carefully the multicurve with a pencil and a ruler, then take a black marker and, without lifting the marker, follow the line drawn in pencil until it closes up. To conclude check if there is some pencil left of if all has been covered by the marker}. In the first case the multicurve was not connected, in the latter it was. 
\end{remark}

\subsection*{(R1) Homotope a $\CT$-transversal curve into $\CT$-normal form}
Take your curve $\gamma$ and check if there are unallowed arcs of intersection with 2-simplexes of $\CT$, that is a subarc $I$ of $\gamma$ contained in a 2-simplex $\Delta$ in $\CT$ and whose endpoints $\D I$ belong to the same side of $\Delta$. If you find one such unallowed arc then homotope it away into the adjacent 2-simplex. Repeat until there are no more unallowed arcs. Once there are no unallowed arcs we are, by definition, in $\CT$-normal form. This process is polynomial in $\vert\gamma\cap\CT\vert$.

\subsection*{(R2) Pass to/from $\CT$-normal curves from/to words in the generators} 
Let $X\subset\Sigma$ be the dual graph of $\CT$, adding if necessary an arc from the base point $*$ to $X$. Take $T\subset X$ a maximal tree in $X$. Each (oriented) edge in $X\setminus T$ represents an element in $\pi_1(\Sigma,*)$. Write it in terms of the given generators. Now, represent the free homotopy class of the your given curve $\gamma$ in $\CT$-normal form by a curve in $X$. Write the associated word. This process is polynomial in $\vert\gamma\cap\CT\vert$.

Conversely, a word $\gamma$ in the generators is a particular kind of curve in $\Sigma$. Run (R1) to freely homotope it to a curve in $\CT$-normal form. This process is polynomial in $\vert\gamma\vert$.

\begin{remark}
There are many algorithms dealing with words in the generating set, or at least where the input is given via words. In the light of (R2) we can make use of all those algorithms. We also get that, when it comes to the complexity, it does not matter if we consider $\vert\gamma\cap\CT\vert$ or $\vert\gamma\vert$. Unless we want to stress which one is the right one, we will just say {\em the complexity of} $\gamma$.
\end{remark}

\subsection*{(R3) Check if elements are trivial or equal}
Surface groups are hyperbolic. This implies that they have linear time solvable word problem, meaning that it takes linear time in the complexity of $\gamma$ to detect if a word $\gamma$ represents the trivial element in $\pi_1(\Sigma,*)$. In particular it takes also linear time to decide if two words represent the same element in $\pi_1(\Sigma,*)$. 

\subsection*{(R4) Check if elements are conjugate}
Surface groups have, again because they are hyperbolic, a linear time solvable conjugacy problem \cite{Holt}. This means that it takes linear time in the complexity of the curves $\gamma$ and $\eta$ to decide if they are freely homotopic to each other or not.

\subsection*{(R5) Compute intersection numbers}
There are many algorithms to compute the intersection number $\iota(\gamma,\eta)$ between curves $\gamma,\eta$. A rather effective one is due to Despr\'e and Lazarus \cite{Despre-Lazarus}. Their algorithm takes quadratic time in the complexity of the curves.

\subsection*{(R6) Decide if two triangulated surfaces are homeomorphic by a homeomorphism with prescribed boundary values}
Suppose that we are given two triangulated oriented surfaces $S_1,S_2$ and that we have a bijection $\pi_0(\D S_1)\to\pi_0(\D S_2)$. We want to figure out if there is a homeomorphism $S_1\to S_2$ inducing the given bijection. Well, we start by determining if there is a bijection $\pi_0(S_1)\to\pi_0(S_2)$ compatible with the one given between the sets of boundary components. If this is not the case then we are done. Otherwise compare the Euler-characteristic of each component of $S_1$ with that of the corresponding component of $S_2$. If all those numbers agree then there is a homeomorphism and of not then not. This process is linear in the size of the triangulations, that is for example in the number of 2-simplices.

\subsection*{(R7) If possible, give a proper orientation preserving homotopy equivalence (with prescribed boundary values) between two triangulated surfaces}

As in (R6) suppose that we are given two triangulated oriented surfaces $S_1,S_2$ and consider $\D S_1$ and $\D S_2$ with their induced orientations. Fix also a bijection $\pi_0(\D S_1)\to\pi_0(\D S_2)$. If possible, we want to construct a homotopy equivalence $S_1\to S_2$ which maps each component of $\D S_1$ homeomorphically to the corresponding component of $\D S_2$. Recall that such a homotopy equivalent exists if and only if there is a homeomorphism. We run (R6) to check if this if this is the case. If not then we dedicate our time to better things. If yes, what we actually get out of (R6) is a bijection between connected components of $S_1$ and $S_2$ compatible with the bijection of boundary components and which maps components of $S_1$ to homeomorphic components of $S_2$. Said shorter, we might assume from now on that $S_1$ and $S_2$ are connected and homeomorphic. We also barycentrically subdivide the triangulations a few times to avoid degenerate cases, $10\vert\chi(S_1)\vert$-times definitively suffice. 

We start by removing an open 2-simplex $\Delta_i$ from each one of the surfaces $S_i$. Starting on the new boundary component $\D\Delta_i$ we can collapse 2-cells one by one, obtaining thus a retraction of each surface $S_i\setminus\Delta_i$ to a graph $X_i$ contained in the 1-skeleton and with $\D S_i\subset X_i$. Let also $\hat\Delta_i$ be the disk obtained when we cut $S_i$ along $X_i$; it is easy to give a homeomorphism between $\Delta_i$ and $\hat\Delta_i$.

Choose now maximal trees $T_1\subset X_1$ and $T_2\subset X_2$ such that the intersection of $T_i$ with each one of the boundary components of $S_i$ is connected. Note that each homeomorphism $X_1/T_1\to X_2/T_2$ which maps the edge of $X_1/T_1$ corresponding to some component of $\D S_1$ to the corresponding edge of $X_2/T_2$ can be lifted to a homotopy equivalence $X_1\to X_2$ preserving the bijection between boundary components. To decide if this homotopy equivalence extends to a homotopy equivalence $S_1\to S_2$ we just need to check that it extends over to $\hat\Delta_1$. This is the case if and only if the curve $\D\hat\Delta_1$ is homotopically trivial in $S_2$ -- use (R3) to check this. If the homotopy extends it then extend it (we leave the reader the details how to do that). Otherwise try a different, by what we mean non-homotopic, homeomorphism $X_1/T_1\to X_2/T_2$ which maps the edge of $X_1/T_1$ corresponding to some component of $\D S_1$ to the corresponding edge of $X_2/T_2$. There are finitely many homotopy classes of homeomorphisms to check and one of them works, so run them all if needed. The complexity of this process if polynomial in the sizes of the given triangulations. 

\begin{remark}
Note that the bound for the complexity also implies that other quantities, for example the Lipschitz constant, associated to the obtained proper homotopy equivalence are also polynomial in the complexity of the input. Recall also that by the Baer-Dehn-Nielsen theorem \cite{MR2850125}, every proper homotopy equivalence $\Sigma\to\Sigma$ is properly homotopic to a homeomorphism, and that every two properly homotopic homeomorphisms are homotopic via homeomorphism. This means that every proper homotopy equivalence $\Sigma\to\Sigma$ represents a unique mapping class. 
\end{remark}

\begin{remark}
If instead of merely a bijection $\pi_0(\D S_1)\to\pi_0(\D S_2)$ we actually have a homotopy class $[\D S_1 \to \D S_2]$ of homotopy equivalences, we can also detect if there is a homeomorphism from $S_1$ to $S_2$ inducing the given homotopy class of boundary map; and if it exists, we can construct one such homeomorphism.
All of this in polynomial time.
We leave the details to the reader.
\end{remark}

\subsection*{(R8) Construct a triangulation adapted to a simple topological multicurve}
Suppose that we are given a simple multicurve $\gamma=\gamma_{\vec m}$ in terms of the admissible vector $\vec m\in\CA(\CT)\subset \BZ_{\ge 0}^{E(\CT)}$. Draw the multicurve and cut each 2-simplexes of $\CT$ along $\gamma$. In this way we get $\vert\CT\vert+\ell_\CT(\gamma)$ pieces. Among these pieces, leave the triangles as they are and add diagonals to triangulate squares, pentagons and hexagons. We obtain in this way a triangulation which contains $\gamma$ as a simplicial subcomplex. The time complexity of the process is polynomial in the complexities of $\gamma$ and $\CT$.

\subsection*{(R9) Detect if two simple multicurves have the same topological type}
It suffices to discuss the case of ordered multicurves, leaving the unordered case to the reader. Anyways, the first step is to run (R8) to construct triangulations adapted to both simple multicurves. Then use (R6) to decide if there is a homeomorphism of the complement of the first one to the complement of the other one, with boundary values given by the fact that we are mapping the curves in order and in an orientation preserving way. If such a homeomorphisms exists, then also does a homeomorphism $\Sigma\to\Sigma$ mapping the first curve to the second one. And if not, then not. The complexity is polynomial in the complexity of the curves and the triangulations.

\subsection*{(R10) Check if a curve $\gamma$ is filling and, if this is not the case, detect the up to isotopy smallest subsurface $\Sigma(\gamma)$ containing $\gamma$}

As we mentioned in the introduction, Kudlinska gave in \cite{Monika} a polynomial time algorithm detecting if a curve $\gamma$ fills the surface or not. Her algorithm is based on the idea that to check if a curve is filling, it suffices to check if it has positive intersection number with a finite number of curves, where the number depends on the complexity of the original curve. We sketch a small modification of her algorithm, using triangulations instead of Dehn-Thurston coordinates. 

The starting point of Kudlinska's algorithm was the observation that if a hyperbolic geodesic $\gamma$ fails to be filling, then the boundary of the subsurface filled by $\gamma$ is an essential curve disjoint of $\gamma$ and which is shorter than $\gamma$. Here is a combinatorial version of this fact:

\begin{lemma}\label{lemblablabla}
A $\CT$-normal curve $\gamma$ is filling if and only if we have $\iota(\gamma,\eta)>0$ for every $\CT$-normal simple multicurve $\eta$ with $\ell_\CT(\eta)\le 2\ell_{\CT}(\gamma)$. 
\end{lemma}

\begin{remark}
The number $2$ in the statement of Lemma \ref{lemblablabla} can be deleted.
\end{remark}

We get from the lemma that $\gamma$ is filling unless we find an admissible vector $\vec m\in\CA(\CT)\subset\BZ_{\ge 0}^{E(\CT)}$ with norm less than $2\ell_{\CT}(\gamma)$ and such that the submanifold $\gamma_{\vec m}$ satisfying \eqref{I am sick of this} for each $e\in E(\CT)$ is a multicurve with $\iota(\gamma_{\vec m},\gamma)=0$. There are polynomially many vectors satisfying the given bound, and it takes polynomial time to check if $\gamma_{\vec m}$ is a multicurve and if the intersection number vanishes, meaning that we can guarantee in polynomial time that $\gamma$ is filling if this is the case. Also, if $\gamma$ is not filling then we will find in polynomial time an essential simple curve $\eta_1$ disjoint of $\gamma$. Cut $\Sigma$ along $\eta_1$ and let $\Sigma_1$ be the connected component of $\Sigma\setminus\eta_1$ which contains $\gamma$. Run the same process again to detect if $\gamma$ fills $\Sigma_1$. If this is the case then we have $\Sigma(\gamma)=\Sigma_1$. If not, find an essential curve $\eta_2\subset\Sigma_1$ disjoint of $\gamma$, let $\Sigma_2$ be the connected component of $\Sigma_1\setminus\eta_2$ containing $\gamma$, and so on... this process runs as most $3g-3+b$ times and it is polynomial at all steps, meaning that in polynomial time we have found the essential subsurface $\Sigma(\gamma)$ of $\Sigma$ filled by $\gamma$. 

The running time of this process is polynomial in the complexity of $\gamma$ (but not of the triangulation).

\begin{proof}[Proof of Lemma \ref{lemblablabla}]
To begin with let $\gamma'$ be a curve which is freely homotopic to $\gamma$ and has minimal self-intersection. For example, we could obtain such a $\gamma'$ by avoiding monogons and bigons. This means in particular that there is such a $\gamma'$ with 
$$\ell_\CT(\gamma')\le\ell_\CT(\gamma).$$
Now, since $\gamma'$ has minimal self-intersection number, we get that the, up to isotopy, smallest essential subsurface $\Sigma(\gamma')$ containing $\gamma'$ can be obtained as a regular neighborhood of the union of $\gamma'$ with all components $\Sigma\setminus\gamma'$ which are disks or annuli containing a boundary component of $\Sigma$. Note that $\gamma'$, and hence $\gamma$ because they both represent the same homotopy class, fails to be filling if and only if $\eta=\D\Sigma(\gamma')$ is not empty. Note also that in the latter case $\eta$ is a $\CT$-normal simple multicurve with $\ell_\CT(\eta)\le 2\ell_{\CT}(\gamma)$ and satisfying 
$$\iota(\gamma,\eta)=\iota(\gamma',\eta)=0.$$
We are done.
\end{proof}

\begin{remark}
As we mentioned above, Kudlinska used Dehn-Thurston coordinates to parametrise simple curves. She also worked with the hyperbolic length instead of $\ell_\CT(\gamma)$. Although it is a bit more complicated to set up than what we just did, her approach has some advantages and in section \ref{sec: dehn-thurston} we will use both Dehn-Thurston coordinates and hyperbolic metric when we want a more effective bound for the running time of our algorithms. We stress than other than that we just followed Kudlinska's algorithm.
\end{remark}

Anyways, once we have gotten our personal Swiss proficient in these 10 routines we can describe the algorithms needed to prove the main results of this paper.

\section{The algorithms}\label{proof-thm-main}

In this section we prove the two results announced in the introduction.

\begin{proof}[Proof of Theorem \ref{thm-main}]
To prove the theorem we just have to describe the promised algorithm. 
We start by fixing an ordered pants decomposition $\vec P$, a simple multicurve $B$ which fills together with $\vec P$, and a triangulation $\CT$ of $\Sigma$ with respect to which $\vec P$ and $B$ admit simplicial embedded representatives meeting in only $\iota(P,B)$ points.

We will first prove the theorem under the following
\begin{quote}
{\bf additional assumptions}: the center of the mapping class group is trivial and the given elements $\gamma,\eta\in\pi_1(\Sigma,*)$ are primitive.
\end{quote}
We will discuss how to deal with the general case once we have proved the theorem under these additional assumptions.
\medskip

Anyways, we start the proof. Our inputs are primitive elements $\gamma,\eta\in\pi_1(\Sigma,*)$, written as words with respect to some finite generating set. For starters we can assume that these words are not only reduced but also cyclically reduced. Note also that we can use (R3) to check if either $\gamma$ or $\eta$ are simple and that we can use (R9) to detect if simple curves have the same type or not. We assume from now on that neither $\gamma$ nor $\eta$ is simple.

\subsection*{Step 1: Reduce to the filling case} 
We can first run (R10) to decide if the curve $\gamma$ is filling, detecting, if this is not the case, the smallest subsurface $\Sigma(\gamma)\subset\Sigma$ filled by $\gamma$. Run (R8) to get a triangulation of $\Sigma(\gamma)$ and proceed in the same way with $\Sigma(\eta)$ using then (R4) to detect if $\Sigma(\gamma)$ and $\Sigma(\eta)$ are homeomorphic. If this is not the case then \underline{$\gamma$ and $\eta$ are not of the same type}. 

Suppose that $\Sigma(\gamma)$ and $\Sigma(\eta)$ are homeomorphic and list all possible homotopy classes $[\D\Sigma(\gamma) \to \D\Sigma(\eta)]$ of homotopy equivalences. 
Using (R6), decide for each one of these homotopy classes if it is induced by some homeomorphism $\Sigma\to\Sigma$ which maps $\Sigma(\gamma)\to\Sigma(\eta)$. 
If there is no homotopy class of homotopy equivalences for which such a homeomorphism exists then \underline{$\gamma$ and $\eta$ are not of the same type}. 

Let $\CB$ be the list of all homotopy classes of homotopy equivalences $[\D\Sigma(\gamma) \to \D\Sigma(\eta)]$ realised by some homeomorphism of $\Sigma$ mapping $\Sigma(\gamma)$ to $\Sigma(\eta)$, and  for each $\pi\in\CB$ run (R7) to get a proper homotopy equivalence 
$$\phi_\pi:\Sigma(\gamma)\to\Sigma(\eta)$$
inducing $\pi$. Then we are facing the following alternative:
\begin{itemize}
\item If there are $\pi\in\CB$ and a pure mapping class $\psi\in P\Map(\Sigma(\eta))$ such that $\psi(\phi_\pi(\gamma))$ is freely homotopic to $\eta$ then \underline{$\gamma$ and $\eta$ are of the same type}.
\item Otherwise they are \underline{they don't have the same type}
\end{itemize}
We have to figure out, in polynomial time, if such $\pi$ and $\psi$ exist. Since the cardinality of $\CB$ is bounded in terms of $g$ and $b$, and since we can work each $\pi$ individually, it suffices to be able to check in polynomial time if for fixed $\pi\in\CB$ the pure mapping class $\psi$ exists or not. A priori this sounds very similar to the initial problem, replacing only $\gamma$ by $\phi_\pi(\gamma)$ and the mapping class group by the pure mapping class group. It is indeed very similar. However we are now in a situation where the given curves fill the surface. This means that, all things considered, we are reduced to having to give a polynomial time algorithm solving the following problem:
\begin{quote}
{\em Given two primitive filling curves $\gamma$ and $\eta$ in $\Sigma$, determine if there is $\phi\in P\Map(\Sigma)$ with $\phi(\gamma)=\eta$.}
\end{quote}
This is what we will do from now on. As we already mentioned in the introduction, the basic strategy is to produce in polynomial time a collection $\CC\subset P\Map(\Sigma)$ of mapping classes, the candidates, with the following properties:
\begin{itemize}
\item[(C1)] If $\phi\in P\Map(\Sigma)$ is such that $\phi(\gamma)=\eta$, then $\phi\in\CC$.
\item[(C2)] There are polynomially many elements in $\CC$ and it takes polynomial time to decide if $\phi(\gamma)=\eta$ for each one of them individually.
\end{itemize}
Once we have this list then we just have to check for each $\phi\in\CC$ if $\phi(\gamma)=\eta$. If we find one then \underline{$\gamma$ and $\eta$ are of the same type}, otherwise \underline{they aren't}.
\medskip

To produce the desired list $\CC$ of candidates we start by making a preliminary list, this time consisting of pants decompositions:

\subsection*{Step 2: Get the set $\CP$ of plausible images of $\vec P$}
Suppose that $\vec\alpha$ is an ordered pants decomposition satisfying $\vec\alpha=\phi(\vec{P})$ for some $\phi\in P\Map(\Sigma)$ with $\phi(\gamma)=\eta$. Applying Lemma \ref{mainlem} to the filling curve $\eta$ and to $\vec\alpha$ we get
\begin{equation}\label{eq-11}
\ell_\CT(\vec\alpha)
\le \ell_{\CT}(\eta)\cdot\iota(\vec\alpha,\eta)=\ell_{\CT}(\eta)\cdot\iota(\phi(\vec{P}),\phi(\gamma))=\ell_{\CT}(\eta)\cdot\iota(\vec{P},\gamma).
\end{equation}

\begin{remark}
For later use note that, replacing $\vec{P}$ by $B$ in this computation we also get
\begin{equation}\label{eq-10}
\ell_\CT(B')\le\ell_{\CT}(\eta)\cdot\iota(B,\gamma)
\end{equation}
for any multicurve $B'$ which could arise as $B'=\phi(B)$ for some $\phi\in P\Map(\Sigma)$ with $\phi(\gamma)=\eta$.
\end{remark}

Now, for every admissible $\vec m\in\CA(\CT)\subset\BZ_{\ge 0}^{E(\CT)}$ with 
$$\Vert\vec m\Vert\stackrel{\text{def}}=\sum\vert m_e\vert\le\ell_{\CT}(\eta)\cdot\iota(\vec{P},\gamma)$$
we consider the submanifold $\gamma_{\vec m}$ satisfying \eqref{I am sick of this} for each $e\in E(\CT)$ and check if it is a pants decomposition. To do that we have to check if its components are essential (R3), if no two distinct components are freely homotopic to  each other (R4), and if the connected components of the complement are pairs of pants (R6). If $\gamma_{\vec m}$ is not a pants decomposition throw it away while proclaiming it to be dust of the earth. Otherwise, that is if $\gamma_{\vec m}$ is a pants decomposition, take the $2^{3g-3+b}\cdot(3g-3+b)!$ ordered oriented multicurves obtained by considering all possible orderings and orientations of the components of $\gamma_{\vec m}$. For everyone of those ordered oriented multicurves we check if they have the same type as $\vec P$. Those who don't are dumped into the dustbin of history, while we preciously keep the others in the set
$$\CP=\left\{
\vec\alpha\middle\vert\begin{array}{l}
\text{ ordered oriented multicurve of type }\vec P\text{ whose underlying multicurve }\alpha\\\text{ arises as }\gamma_{\vec m}\text{ for some admissible }\vec m\in\CA(\CT)\text{ with }\Vert\vec m\Vert\le\ell_{\CT}(\eta)\cdot\iota(\vec{P},\gamma)
\end{array}\right\}.$$
Note that, by the preceding discussion, it takes polynomial time to produce the set $\CP$ and that \eqref{eq-11} holds for every $\vec\alpha\in\CP$.

So far all we have is a collection of pants decompositions. What we want is a collection of mapping classes of homotopy equivalences. 

\subsection*{Step 3: Get a preliminary list $\CC_0$ of maps}
Take now $\vec\alpha\in\CP$ and run (R8) to get and adapted triangulation. Then we run (R7) on each connected component of $\Sigma\setminus\vec\alpha$ to get a proper homotopy equivalence 
$$\phi_{\vec\alpha}:\Sigma\to\Sigma$$
mapping each boundary component to itself and $\vec P$ to $\vec\alpha$. Let
$$\CC_0=\{\phi_{\vec\alpha}\text{ for }\vec\alpha\in\CP\}$$
be the collection of so-obtained maps. 

As remarked after describing (R7) we can bound the Lipschitz constant of $\phi_{\vec\alpha}\in\CC_0$ polynomially in terms of $\ell_\CT(\vec\alpha)$, meaning in particular that
$$\ell_\CT(\phi_{\vec\alpha}(\beta))\le\Pol(\ell_{\CT}(\alpha))\cdot \ell_\CT(\beta)$$
for every curve $\beta\subset\Sigma$. Taking \eqref{eq-11} into consideration, and reminding the reader that the concrete polynomial $\Pol$ might change from line to line, we get
\begin{equation}\label{eq-12}
\ell_\CT(\phi_{\vec\alpha}(\beta))\le\Pol\left(\ell_{\CT}(\eta),\iota(\vec{P},\gamma),\ell_\CT(\beta)\right)
\end{equation}
for every $\phi_{\vec\alpha}\in\CC_0$ and every curve $\beta$. 

For the convenience of the reader we summarise what we have achieved after running Step 2 and Step 3:
\begin{quote}
{\bf Outcome so far:} {\em We have produced in polynomial time a collection $\CC_0$ of polynomially many homotopy equivalences $\phi_{\vec\alpha}:\Sigma\to\Sigma$ with the following properties:
\begin{itemize}
\item If $\phi\in P\Map(\Sigma)$ is such that $\phi(\gamma)=\eta$, then there is $\phi_{\vec\alpha}\in\CC_0$ with $\phi(\vec P)=\phi_{\vec\alpha}(\vec P)$.
\item The inequality \eqref{eq-12} holds for every $\phi_{\vec\alpha}\in\CC_0$ and every curve $\beta\subset\Sigma$. 
\end{itemize}}
\end{quote}
As we already mentioned earlier, we get from the Dehn-Baer-Nielsen theorem that every proper homotopy equivalence $\Sigma\to\Sigma$ represents a mapping class. We can thus think of $\CC_0\subset P\Map(\Sigma)$. The collection $\CC_0$ is going to be contained in our definitive list of candidates.

\subsection*{Step 4: Get our list $\CC\subset P\Map(\Sigma)$ of candidates} 

Suppose that we have $\phi\in P\Map(\Sigma)$ with $\phi(\gamma)=\eta$ and let $\phi_0\in\CC_0$ be the element satisfying 
$$\phi(\vec P)=\phi_0(\vec P).$$
This means that $\phi$ and $\phi_0$ differ by an element of the stabiliser of the ordered oriented pants decomposition $\vec P$, that is by a multi-twist along $\vec P$. In other words there is an $\vec n\in\BZ^{3g-3+b}$ with 
\begin{equation}\label{eating curry}
\phi=\phi_0\circ T_{\vec P}^{\vec n}
\end{equation}
where $T_{\vec P}^{\vec n}=T_{P_1}^{n_1}\circ\dots\circ T_{P_{3g-3+b}}^{n_{3g-3+b}}$ is the multi-twist along $\vec P$ which twists $n_i$ times around the component $P_i$. Note that we can rewrite $\phi$ also as 
$$\phi=\phi_0\circ T_{\vec P}^{\vec n}=\phi_0\circ T_{\vec P}^{\vec n}\circ\phi^{-1}_0\circ\phi_0=T_{\phi_0(\vec P)}^{\vec n}\circ\phi_0$$
where $T_{\phi_0(\vec P)}^{\vec n}$ is the multi-twist that which twists $n_i$ times around the $i$-th component of $\phi_0(\vec P)$. Since $\phi(\gamma)=\eta$ we get thus that 
$$(T_{\phi_0(\vec P)}^{\vec n}\circ\phi_{0})(\gamma)=\eta.$$
In particular we get from \eqref{eq-10} that 
$$\ell_\CT((T_{\phi_0(\vec P)}^{\vec n}\circ\phi_{0})(B))\le\ell_{\CT}(\eta)\cdot\iota(B,\gamma),$$
and thus, using that $B,P\subset\CT^{(1)}$, that
\begin{equation}\label{eq-13}
\iota(B+P,T_{\phi_0(\vec P)}^{\vec n}(\phi_{0}(B)))\le\ell_{\CT}(\eta)\cdot\iota(B,\gamma).
\end{equation}
We are now ready to make use of the following lemma due to Ivanov \cite[Lemma 4.2]{Ivanov}

\begin{lemma}\label{Ivanov}
Let $\alpha_1,\dots,\alpha_s$ be disjoint simple curves in $\Sigma$ and let $T=T_{\alpha_1}^{n_1}\cdot\ldots\cdot T_{\alpha_s}^{n_s}$ be the multitwist which twists $n_i$ times around $\alpha_i$. Then we have 
$$\iota(T(\delta),\delta')\ge \sum_i(\vert n_i\vert-2)\cdot\iota(\delta,\alpha_i)\cdot\iota(\alpha_i,\delta')-\iota(\delta,\delta')$$
for all simple multicurves $\delta,\delta'\in\Sigma$.
\end{lemma}

Applying this lemma to $T=T_{\phi_0(\vec P)}^{\vec n}$, to $\delta=\phi_0(B)$, and individually to $\delta'=P$ and $\delta'=B$ we get:
$$\iota(T_{\phi_0(\vec P)}^{\vec n}(\phi_0(B)),P+B)
\ge \sum_i(\vert n_i\vert-2)\cdot\iota(\phi_0(B),\phi_0(P_i))\cdot\iota(\phi_0(P_i),P+B)-\iota(\phi_0(B),P+B).$$
Taking into account that 
$$\iota(\phi_0(B),\phi_0(P_i))=\iota(B,P_i)\ge 1\text{ and }\iota(\phi_0(P_i),P+B)\ge 1$$
we get the rather coarse estimate 
$$\iota(T_{\phi_0(\vec P)}^{\vec n}(\phi_0(B)),P+B)
\ge \sum_i(\vert n_i\vert-2)-\iota(\phi_0(B),P+B).$$
Using \eqref{eq-12}, \eqref{eq-13}, and some algebra we get:
\begin{equation}\label{eq-14}
\sum_i\vert n_i\vert\le \Pol(\ell_{\CT}(\eta),\ell_{\CT}(\gamma)).
\end{equation}
We thus have established that any $\phi\in P\Map(\Sigma)$ with $\phi(\gamma)=\eta$ belongs to the set
$$\CC=\{\phi_0\circ T_{\vec P}^{\vec n}\text{ where }\phi_0\in\CC_0\text{ and }\vec n\in\BZ^{3g-3+b}\text{satisfies \eqref{eq-14}}\}.$$
This means that the set $\CC$ satisfies property (C1). It also satisfies property (C2). With this we are done with the construction of our algorithm and thus with the proof of Theorem \ref{thm-main}... under the additional assumptions that the center of the mapping class group $\Map(\Sigma)$ is trivial and that the elements $\gamma$ and $\eta$ are primitive. We deal with these two issues individually.

\subsubsection*{What to do if the center of $\Map(\Sigma)$ is not trivial?} First one should make clear where did we use this assumption. Well, if the center is not trivial then we do not get \eqref{eating curry}, but rather only that 
$$\phi=\phi_0\circ T^{\vec n}_{\vec P}\circ\tau$$
where $\tau\in C(\Map(\Sigma))$. This is not tragic because, as we mentioned earlier, there is at most a single non-trivial element in $C(\Map(\Sigma))$, namely the hyper-elliptic involution $\tau_0$. It follows that we just have to check if there is some $\vec n\in\BZ^N$ with 
$$\text{either }(\phi_0\circ T^{\vec n}_{\vec P})(\gamma)=\eta\text{ or }(\phi_0\circ T^{\vec n}_{\vec P})(\gamma^{\tau_0})=\eta$$
where $\gamma^{\tau_0}=\tau_0(\gamma)$. This means that we just have to run the algorithm above twice, once for $\gamma$ and once of $\gamma^{\tau_0}$. It is probably unnecessarily to say so, but twice a polynomial is still a polynomial. We have dropped the assumption that $C(\Map(\Sigma))$ is trivial.

\subsubsection*{What to do with elements which are not primitive?}
Let us drop now the assumption that the elements $\gamma$ and $\eta$ are primitive. In \cite{Despre-Lazarus}, Despr\'e and Lazarus show that in surface groups the primitive positive root of an element can be computed in polynomial time. This means that in polynomial time we get $\bar\gamma,\bar\eta\in\pi_1(\Sigma,*)$ primitive and $r,s\ge 1$ with 
$$\gamma=\bar\gamma^r\text{ and }\eta=\bar\eta^s.$$
Now, since (non-trivial) abelian subgroups of surface groups are contained in unique maximal cyclic subgroups, we get that $\bar\gamma,\bar\eta,r$ and $s$ are unique. It follows that $\gamma$ and $\eta$ are of the same type if and only if the primitive elements $\bar\gamma$ and $\bar\eta$ are of the same type and $r=s$. Since we can check that in polynomial time, we are done.
\end{proof}

Having proved Theorem \ref{thm-main} we come now to the proof of Theorem \ref{thm-main1}. There is not much to say:

\begin{proof}[Proof of Theorem \ref{thm-main1}]
Suppose that we are given two elements $\gamma,\eta\in\pi_1(\Sigma,*)$ in the fundamental group of a closed connected surface, and let $\bar\gamma$ and $\bar\eta$ be their conjugacy classes. The elements $\gamma$ and $\eta$ differ by an automorphism $\phi\in\Aut(\pi_1(\Sigma,*))$ if and only if the classes $\bar\gamma$ and $\bar\eta$ differ by an exterior automorphism $\bar\phi\in\Out(\pi_1(\Sigma,*))$. Now, from the Baer-Dehn-Nielsen theorem \cite{MR2850125}, we get an identification $\Out(\pi_1(\Sigma,*))\simeq\Map(\Sigma)$ between the group of exterior automorphisms and the mapping class group. This identification is compatible with the identification between conjugacy classes in $\pi_1(\Sigma,*)$ and free homotopy classes of curves in $\Sigma$. It follows that $\gamma$ and $\eta$ differ by an automorphism if and only if the two associated (free homotopy classes of) curves $\bar\gamma$ and $\bar\eta$ are of the same type. We can decide that in polynomial time grace to the algorithm used to prove Theorem \ref{thm-main}. Moreover, if $\bar\gamma$ and $\bar\eta$ differ by a mapping class then we get some $\bar\phi\in\Map(\Sigma)=\Out(\pi_1(\Sigma,*))$ with $\bar\phi(\bar\gamma)=\bar\eta$. Referring to a representative of $\bar\phi$ by the same symbol, all it remains is to do is to find $g\in\pi_1(\Sigma,*)$ with $\bar\phi(\gamma)=g\eta g^{-1}$. Since, by the proof of Theorem \ref{thm-main}, the Lipschitz constant of $\bar\phi$ is polynomially bounded by the complexity of $\gamma$ and $\eta$ we get that what we have to do is to find in polynomial time a conjugating element. This well-known to be possible because $\pi_1(\Sigma,*)$ is hyperbolic \cite{Holt}. We are done.
\end{proof}

\section{A slightly more efficient approach}\label{sec: dehn-thurston}

Although clearly possible, it would be painful and frustrating to try to give actual estimates for the running time of the algorithm used to prove Theorem \ref{thm-main}: evidently it would be painful, and it would be frustrating because one would get truly horrible bounds. This is why we sketch now a variation of the same algorithm, using Dehn-Thurston coordinates instead of triangulations, getting an not-that-bad estimate for the running time. For the sake of concreteness we will focus on the most interesting case:

\begin{quote}
{\bf Assumption.} {\em Let $\gamma$ and $\eta$ be primitive filling curves in $\Sigma$.}
\end{quote}
 We refer to \cite{Thesis} for a detailed and improved version of the results in this chapter, with explicit values for the involved constants.
 
Anyways, as all along we are aiming to determine if there is $\phi\in P\Map(\Sigma)$ with $\phi(\gamma)=\eta$. As we did above, we fix an ordered pants decomposition $\vec P = (P_1,\dots,P_{3g-3+b})$, and a simple multicurve $B$ with the property that $\iota(B,P_i)=2$ for all $i=1,\dots,3g-3+b$. We also fix a hyperbolic metric on $\Sigma$ and let 
$$L=\max\{\ell_\Sigma(\gamma),\ell_\Sigma(\eta)\}.$$
In the sequel we will write $\const$ for any constant which only depends on $P$, $B$ and the metric, but it might change from line to line.

Anyways, we denote by 
$$\Pi=\Pi_{\vec P,B}:\calD\to\CM\CL_\BZ$$
the Dehn-Thurston parametrisation \cite{Penner-Harer} associated to $\vec P$ and $B$. Here $\CM\CL_\BZ(\Sigma)$ is the set of all integrally weighted simple multicurves and $\calD$ is the set of all pairs $(\vec m,\vec t)\in\BZ_{\ge 0}^{3g-3+b}\times\BZ^{3g-3+b}$ such that $t_i\ge 0$ whenever $m_i=0$ and that $m_{i_1}+m_{i_2}+m_{i_3}$ is even whenever the component $P_{i_1},P_{i_2}$ and $P_{i_3}$ bound a pair of pants. If $\gamma=\Pi(\vec m,\vec t)$ then we let
$$\ell_{\vec P,B}(\gamma) := \sum_{i=1}^{N} m_i + \sum_{i=1}^{N} |t_i|$$
be the \textit{combinatorial length} of a multicurve $\gamma$ with respect to $\vec P$ and $B$.

It follows from general principles (for example the Milnor-\v Svarc lemma), that the combinatorial length of a curve, its hyperbolic length, as well as its intersection number with $\vec P+B$ are all comparable to each other. We thus have:

\begin{lemma}\label{lot of noise1}
With notation as above we have
$$\frac{1}{\const} \cdot\ell_{\vec P,B}(\gamma) \le \iota(\gamma,\vec P+B) \le\const\cdot\ell_{\vec P,B}(\gamma)$$
for every simple multicurve $\gamma$.\qed
\end{lemma}

We also have the following version of Lemma \ref{mainlem}.

\begin{lemma}\label{lot of noise2}
Let $\gamma$ be a filling curve and $\alpha$ be a set of curves in $\Sigma$. Then, we have
$$i(\alpha, \beta) \le 
\const \cdot \ell_\Sigma(\alpha) \cdot \ell_\Sigma(\gamma) \cdot (\iota(\beta,\beta)+1)\cdot \iota(\beta,\gamma)$$
for any multicurve $\beta$ in $\Sigma$.
If moreover $\Sigma$ is closed then the statement remains true even after deleting the factor $1+\iota(\beta,\beta)$.\qed
\end{lemma}

Suppose now that $\phi\in P\Map(\Sigma)$ is such that $\phi(\gamma)=\eta$. Recalling that the value of the constant may change from line to line, we get from these two lemmas that
\begin{align*}
\ell_{\vec P,B}(\phi(\vec P))
&\le\const\cdot\iota(\phi(\vec P),\vec P+B)\\
&\le \const \cdot \ell_\Sigma(\vec P+B) \cdot \ell_\Sigma(\eta) \cdot \iota(\phi(\vec P),\eta)\\
&=\const \cdot \ell_\Sigma(\eta) \cdot \iota(\phi(\vec P),\phi(\gamma))\\
&\le \const \cdot L \cdot \iota(\vec P,\gamma)\\
&\le \const\cdot L^2.
\end{align*}
This means that
$$\phi(\vec{P})\in\CP\stackrel{\text{def}}=\left\{\vec\alpha\ 
\middle\vert\begin{array}{l}
\text{ ordered oriented multicurve of type } \vec P\text{ whose underlying multicurve}\\\text{ arises as }\Pi(\vec{m},\vec{t})
\text{ for some } (\vec{m},\vec{t})\in\calD \text{ with } \Vert (\vec{m},\vec{t}) \Vert \le\const\cdot L^2
\end{array}\right\}.$$
How long does it take to produce $\CP$? Well, for any vector $(\vec m,\vec t)\in\calD$ with $\Vert(\vec m,\vec t)\Vert\le\const \cdot {L^2}$ it takes $O(L^2)$ time to check if the associated multicurve $\alpha=\Pi_{\vec P,B}(\vec m,\vec t)$ is a pants decomposition and, if this is the case, to draw a embedded copy of the dual graph. Once we have the dual graph we can determine, for all possible orders $\vec\alpha$ of the components of $\alpha$, if $\vec\alpha$ is of type $\vec P$: just compare the two dual graphs. But even more, knowing how to map $\vec P$ and its dual graph to $\vec\alpha$ and its dual graph, we actually get a concrete homotopy equivalence doing exactly that.

This means that it take $O(L^2)$ time to check if some $(\vec m,\vec t)\in\calD$ with $\Vert(\vec m,\vec t)\Vert\le\const\cdot L^2$ corresponds to $\vec\alpha\in\CP$ and, if this is the case, to construct a concrete proper homotopy equivalence $\phi_{\vec\alpha}$ sending $\vec P$ to $\vec\alpha$. Since there are at most $L^{4(3g-3+b)}$ vectors to check we get:

\begin{lemma}
It takes time $O(L^{12g-10+4b})$ to do (1) produce $\CP$, and (2) for each $\vec\alpha\in\CP$ produce a proper homotopy equivalence  $\phi_{\vec\alpha}:\Sigma\to\Sigma$ mapping $P$ to $\vec\alpha$ satisfying
$$\iota(\phi_{\vec\alpha}(\beta),\vec P+B) \le \const \cdot \iota(\beta,\vec P+B) \cdot L^2$$
for every simple multicurve $\beta$. \qed
\end{lemma}

Still assuming that $\phi(\gamma)=\eta$ we note that we have
$$\phi=\phi_{\phi(\vec P)}\circ T^{\vec n}_{\vec P}$$
for some multitwist $T^{\vec n}_{\vec P}$ along $\vec P$. Arguing as above, that is, using Ivanov's lemma we get a bound on the norm of $\vec n$:
$$\Vert\vec n\Vert\le\const\cdot L^2.$$
On the other hand, if we are given $\vec\alpha\in\CP$ and $\vec n \in \BZ^{3g-3+b}$ with at most norm $\const\cdot L^2$, it takes $O(L^2)$ to check if $(\phi_{\vec\alpha}\circ T_{\vec P}^{\vec n})(\gamma)=\eta$. Taking all of this together we get:

\begin{quote}
{\em It takes $O(L^{18g+6b-14})$ to detect if $\gamma$ and $\eta$ are of the same type.}
\end{quote}

We are done.

\bibliographystyle{amsplain}
\bibliography{references}

\end{document}